\documentclass[12pt]{amsart}

\usepackage{amssymb,amsfonts}
\usepackage{float}
\usepackage{graphicx}
\usepackage[all,arc]{xy}
\usepackage{enumerate}
\usepackage{mathrsfs}
\usepackage{tikz-cd}
\usepackage[left=1in,top=1in,right=1in,bottom=1in]{geometry}
\usepackage[colorlinks=true, allcolors=blue]{hyperref}

\usepackage{fullpage}

\newtheorem{thm}{Theorem}[section]
\newtheorem*{thm*}{Theorem}

\newtheorem{innercustomthm}{Theorem}
\newenvironment{customthm}[1]
  {\renewcommand\theinnercustomthm{#1}\innercustomthm}
  {\endinnercustomthm}

\newtheorem{prop}[thm]{Proposition}
\newtheorem{lem}[thm]{Lemma}

\theoremstyle{definition}
\newtheorem{defn}[thm]{Definition}

\newtheorem{rem}[thm]{Remark}

\newcommand{\N}{\mathbf{N}}
\newcommand{\Z}{\mathbf{Z}}
\newcommand{\Q}{\mathbf{Q}}
\renewcommand{\P}{\mathbf{P}}

\newcommand{\Hom}{\mathrm{Hom}}
\newcommand{\C}{\mathbf{C}}
\newcommand{\trop}{\text{trop}}

\newcommand{\R}{\mathbf{R}}

\newcommand{\cat}{\mathsf}
\newcommand{\pr}{\mathrm{pr}}
\renewcommand{\Pr}{\mathrm{Pr}}

\makeatletter
\let\c@equation\c@thm
\makeatother
\numberwithin{equation}{section}

\graphicspath{ {thesis-images/} }

\title{Chow rings of heavy/light Hassett spaces via Tropical Geometry}

\author{Siddarth Kannan}\address{Siddarth Kannan, Department of Mathematics, Brown University, Providence, RI 02912}\email{siddarth\_kannan@brown.edu}
\author{Dagan Karp}\address{Dagan Karp, Department of Mathematics, Harvey Mudd College, Claremont, CA 91711}\email{dagan.karp@hmc.edu}
\author{Shiyue Li}\address{Shiyue Li, Department of Mathematics, Brown University, Providence, RI 02912}\email{shiyue\_li@brown.edu}
\date{\today}

\usepackage{microtype} 
\begin{document}

	\begin{abstract}
		We compute the Chow ring of an arbitrary \textit{heavy/light Hassett space} $\overline{M}_{0, w}$. These spaces are moduli spaces of weighted pointed stable rational curves, where the associated weight vector $w$ consists of only \textit{heavy} and \textit{light} weights. Work of Cavalieri et al. ~\cite{CHMR} exhibits these spaces as tropical compactifications of hyperplane arrangement complements. The computation of the Chow ring then reduces to intersection theory on the toric variety of the Bergman fan of a graphic matroid. Keel ~\cite{Keel} has calculated the Chow ring $A^*(\overline{M}_{0, n})$ of the moduli space $\overline{M}_{0, n}$ of stable nodal $n$-marked rational curves; his presentation is in terms of divisor classes of stable trees of $\P^1$'s having one nodal singularity. Our presentation of the ideal of relations for the Chow ring $A^*(\overline{M}_{0, w})$ is analogous. We show that pulling back under Hassett's birational reduction morphism $\rho_w: \overline{M}_{0, n} \to \overline{M}_{0, w}$ identifies the Chow ring $A^*(\overline{M}_{0, w})$ with the subring of $A^*(\overline{M}_{0, n})$ generated by divisors of \textit{$w$-stable trees}, which are those trees which remain stable in $\overline{M}_{0, w}$.
	\end{abstract}
	\maketitle
	\section{Introduction}	
	In ~\cite{Hassett}, Brendan Hassett introduced moduli spaces of $\Q$-weighted pointed stable curves, in the context of the log minimal model program. 
	For a pair $(g,w)$, where $g$ is a nonnegative integer and $w = (w_i) \in \left(\Q \cap (0, 1] \right)^n$ is a rational weight vector such that $2g - 2 + \sum w_i > 0$, he constructs smooth Deligne-Mumford stacks $\overline{\mathcal{M}}_{g, w}$, with corresponding coarse projective schemes $\overline{M}_{g, w}$, as alternate modular compactifications of the stack $\mathcal{M}_{g, n}$ (resp. $M_{g, n}$), which parameterizes smooth curves of genus $g$ with $n$ distinct marked points. The moduli space $\overline{\mathcal{M}}_{g, w}$ parameterizes pairs $(C, D)$ where $C$ is an at worst nodal curve of arithmetic genus $g$, and $D = \sum w_i P_i$ is a $\Q$-divisor weighted by $w$, such that $K_C + D$ is ample along each irreducible component of $C$. For certain choices of weight vector $w$, these spaces arise in the study of moduli of stable quotients ~\cite[Section 4]{marian2011moduli}, and when $g = 0$ and $w = (1, 1, \epsilon, \ldots, \epsilon)$, the Hassett space  coincides with the Losev-Manin modular compactification of $M_{0, n}$ ~\cite{LosevManin}.
	
	In this note we are concerned solely with the case $g = 0$, where the spaces $\overline{M}_{0, w}$ are smooth projective fine moduli schemes, parameterizing marked \textit{$w$-stable trees of $\P^1$'s}: the required ampleness of the divisor $K_C + D$ along each component can be reinterpreted as in the following definition.

	\begin{defn}\label{Trees}
		Fix an algebraically closed field $k$ and a vector of rational weights ${w \in (\Q \cap (0, 1])^n}$ satisfying $\sum w_i > 2$. For a $k$-scheme $B$, a \textit{family of nodal $w$-stable curve of genus $0$ over $B$} is a flat proper morphism $\pi: \mathcal{C} \to B$, together with $n$ sections $s_1, \ldots, s_n$, such that
		\begin{enumerate}[(a)]
			\item each geometric fiber $\mathcal{C}_b$ of $\pi$ is a reduced, connected curve of arithmetic genus $0$, the singularities of which are ordinary double points, called \textit{nodes};
			\item $s_{i_1}(b) = s_{i_2}(b) = \cdots = s_{i_p}(b)$ implies that \[\sum_{j = 1}^{p} w_{i_j} \leq 1;\]
			\item for an irreducible component $T \subseteq \mathcal{C}_b$, we have
			\[ \# \{\text{nodes on }T \} + \sum_{j \in J(T)} w_j > 2, \] 
			where $J(T) = \{j: s_j(b)\in T\}$.
		\end{enumerate}
		A curve appearing as the fiber $\mathcal{C}_b$ above is called a \textit{$w$-stable tree of $\P^1$'s}. 
	\end{defn}
	If we let $\overline{\mathcal{M}}_{0, w}: \cat{Sch}_k^{\cat{op}} \to \cat{Set}$ be the functor which takes a $k$-scheme $B$ to the set of $w$-stable curves over $B$, then $\overline{\mathcal{M}}_{0, w}$ is represented by a smooth projective scheme $\overline{M}_{0, w}$. When ${w = (1^n)}$, we recover the space $\overline{M}_{0, n}$ of $n$-marked stable rational curves, whose intersection theory was first computed by Keel ~\cite{Keel}, and then by Tavakol ~\cite{tavakol2017chow}; both use a description of $\overline{M}_{0, n}$ as an iterated blowup of $(\P^1)^{n - 3}$. 
	\begin{thm}[\cite{Keel}]\label{Keel}
		Let $n \ge 4$.
		The Chow ring of $\overline{M}_{0, n}$ is given as
		\[ A^*(\overline{M}_{0, n}) \cong \frac{ \Z[D^S \mid S \subseteq \{1, \ldots, n\}, |S|, |S^c| \geq 2]}{\langle \text{the following relations} \rangle},   \]
		\begin{enumerate}
			\item $D^T - D^{T^c} = 0$ for all $T$.
			\item $D^T D^S = 0$ unless one of the following holds:
			\[S \subseteq T,\, T \subseteq S, \, S^c \subseteq T\,, T^c \subseteq S. \]
			\item For any four distinct elements $i, j, k, \ell \in \{1, \ldots, n \}$,
			\[\sum_{\substack{i, j \in S \\ k, \ell \notin S }} D^S = \sum_{\substack{i, k \in S \\ j, \ell \notin S }} D^S = \sum_{\substack{i, \ell \in S \\ j, k \notin S }} D^S.\]
		\end{enumerate}	
	\end{thm}
	We derive an analogous presentation for the spaces of weighted stable curves $\overline{M}_{0, w}$ whenever $w$ is \textit{heavy/light}, meaning that up to isomorphism, $\overline{M}_{0, w}$ is given by a weight vector ${w = (1^m, \epsilon^{(n - m)})}$ where $\epsilon < \frac{1}{n - m}$ and $m \geq 2$. These are the Hassett spaces whose intersection theory can be computed using tools from tropical geometry, in a sense to be described later. We prove the following theorem. 
	\begin{thm}\label{Main}
		Let $m \geq 2$ and $n \geq 4$, and suppose that $w$ is a heavy/light weight vector, with $m$ heavy and $(n - m)$ light weights. Then the Chow ring of $\overline{M}_{0, w}$ is given as follows:
		\[A^*(\overline{M}_{0, w}) = \frac{\Z\left[D^S \mid S \subsetneq \{2, \ldots, n \},\, \sum_{i \in S} w_i > 1\right]}{\langle \text{the following relations} \rangle}, \]
		\begin{enumerate}
			\item $D^S D^T = 0$ unless one of the following holds: $S \subseteq T$, $T \subseteq S$, $S \cap T = \varnothing$.
			\item For any pair of two-element subsets $\{i, j\}, \{k, \ell\} \subseteq \{2, \ldots, n \}$ with $i, k \leq m$, we have the linear relation
			\[ \sum_{\substack{S \not\supseteq \{k, \ell \}\\ S \supseteq \{i, j\}}} D^S = \sum_{\substack{S \supseteq \{k, \ell\} \\ S \not\supseteq \{ i, j \}}} D^S. \]
		\end{enumerate}
	\end{thm}	
	The Chow ring is generated by divisors, which correspond to the combinatorial types of $w$-stable trees: for $S \subsetneq \{2,\ldots, n\}$ with $\sum_{i \in S} w_i > 1$, the divisor $D^S \subseteq \overline{M}_{0,w}$ is the closure of the locus of nodal curves with two components, one of which is marked by $S$, and the other marked by $S^c \cup \{1\}$. The multiplicative relation in Theorem \ref{Main} then follows from the fact that the divisors $D^S$ and $D^T$ are disjoint unless $S \subseteq T$, $T\subseteq S$, or $S \cap T = \varnothing$.
	
	\begin{rem}
	An alternative, but equivalent set of linear relations to those in Theorem \ref{Main} can be obtained by exploiting forgetful maps between heavy/light Hassett spaces. Indeed, given distinct indices $i, j, k, \ell \in \{2, \ldots, n\}$ such that $i, k \leq m$, one gets a map
	\[\pi_{i, j, k, \ell}: \overline{M}_{0, w} \to \overline{M}_{0, (w_i, w_j, w_k, w_\ell)} \]
	by forgetting all marked points except those indexed by $i$, $j$, $k$, and $\ell$.
	Since $\overline{M}_{0,(w_i, w_j, w_k, w_\ell)}$ is a proper nonsingular curve which admits a birational morphism from $\overline{M}_{0, 4} \cong \P^1$, we have $\overline{M}_{0,(w_i, w_j, w_k, w_\ell)} \cong \P^1$. Then linear relations in the Chow ring of $\overline{M}_{0, w}$ may be obtained by pulling back the linear equivalence between any two points on $\P^1$ along $\pi_{i, j, k, \ell}$; this is how one obtains Keel's linear relations in Theorem \ref{Keel}. The linear relations in Theorem \ref{Main} are the ones arising naturally from matroid theory and toric geometry. 
	\end{rem}
	
	To better relate our presentation to Keel's via the reduction morphisms introduced by Hassett, we also prove the following theorem.
	\begin{thm}
		Let $w$ be a heavy/light weight vector, and let $\rho_w: \overline{M}_{0, n} \to \overline{M}_{0, w}$ denote the corresponding birational reduction morphism. Then the pullback
		\[\rho_w^*: A^*(\overline{M}_{0, w}) \to A^*(\overline{M}_{0, n}) \]
		identifies the Chow ring of $\overline{M}_{0, w}$ with the subring of $A^*(\overline{M}_{0, n})$ generated by the divisor classes of $w$-stable trees.
	\end{thm}
	We note that in principle, Theorem ~\ref{Main} follows from Petersen's ~\cite{petersen2017} work on Chow rings of Fulton-MacPherson compactifications, but our method is completely different, exhibiting an application of tropical geometry.
	\subsection{Tropical moduli spaces of weighted stable rational curves.} The relationship between $\overline{M}_{0, n}$ and its tropical counterpart $M_{0, n}^\trop$ is a beautiful example of tropical mathematics as a bridge between algebraic geometry and combinatorics, particularly graph and matroid theory. The space $M_{0, n}^\trop$ has a concrete description as a cone complex, which parameterizes metrics on dual graphs of singular curves in the boundary $\overline{M}_{0, n} \setminus M_{0, n}$. This space is also the well-known space of phylogenetic trees, which can be realized as the Bergman fan of the graphic matroid defined from the complete graph on $n - 1$ vertices ~\cite{Ardila}. The cone complex $M_{0, n}^\trop$ is also a fine moduli space in the categorical sense: $M_{0, n}^\trop$ represents a functor taking a rational polyhedral cone complex $\Sigma$ to the set of families $C \to \Sigma$ of $n$-marked tropical curves of genus $0$ over $\Sigma$; see ~\cite{francois_hampe_2013}, ~\cite{CCUW} for a discussion of this perspective.\\
\indent 	While the algebro-geometric moduli space $\overline{M}_{0, n}$ and the combinatorial moduli space $M_{0, n}^\trop$ may be defined independently of one another, one can recover $M_{0, n}^\trop$ as the cone over the dual intersection complex of the simple normal crossings (snc) compactification $M_{0, n} \subseteq \overline{M}_{0, n}$; see ~\cite{Chan} for a detailed discussion of this perspective in the more general setting of the spaces $M_{g, n}^\trop$. The procedure of \textit{geometric tropicalization} gives a way to embed the cone complex $M_{0, n}^\trop$ as a balanced fan in a real vector space. In this framework, one begins with a snc compactification $U \subseteq Y$; put $\Sigma$ for the associated cone complex which is given as the cone over the dual intersection complex. When the complement of $U$ in $Y$ has enough invertible functions, there is an associated embedding $\iota: U \hookrightarrow T$ of $U$ into a torus. The map $\iota$, together with work of Cueto ~\cite{Cueto}, can then be used to embed $\Sigma$ as a balanced fan in a real vector space. This gives a toric variety $X(\Sigma)$ with dense torus $T$, and taking the closure $\overline{\iota(U)}$ gives a compactification of $U$. See ~\cite{MStropical} for a detailed treatment of geometric tropicalization.\\
	 \indent In ~\cite{GMequations}, Gibney and Maclagan carry out the process of geometric tropicalization for the snc pair $M_{0, n} \subseteq \overline{M}_{0, n}$. Write $\Sigma_{n} = B(K_{n - 1})$ for the Bergman fan of the complete graph $K_{n - 1}$; then $\Sigma_{n}$ is the space of phylogenetic trees, embedding the cone complex $M_{0, n}^\trop$ as a balanced fan. Gibney and Maclagan identify $M_{0, n}$ with a quotient by $T^n$ of the open subset $\mathrm{Gr}^0(2, n)$ of the Grassmannian $\mathrm{Gr}(2, n)$ given by the nonvanishing of the Pl{\"u}cker coordinates. They prove the following theorem.
	\begin{thm}[\cite{GMequations}]
		The tropicalization of the torus embedding
		\[M_{0, n} \hookrightarrow T^{\binom{n}{2}}/T^n \]
		gives the fan $\Sigma_{n}$, whose underlying cone complex is identified with $M_{0, n}^\trop$. The toric variety $X(\Sigma_{n})$ has torus $T^{\binom{n}{2}}/T^n$, and taking the closure of $M_{0, n}$ in this toric variety yields the Knudsen-Mumford compactification $\overline{M}_{0, n}$.
	\end{thm}

Ulirsch ~\cite{Ulirsch} showed that while $M_{0, n} \subseteq \overline{M}_{0, w}$ is in general \textit{not} an snc compactification, the pair $M_{0, w} \subseteq \overline{M}_{0, w}$ is snc, where $M_{0, w}$ denotes the locus of smooth, not necessarily distinctly marked curves. Motivated by this observation and the work of Gibney-Maclagan, Cavalieri et al. ~\cite{CHMR} run geometric tropicalization for the snc compactification ${M_{0, w} \subseteq \overline{M}_{0, w}}$. On the tropical side, the authors show that the cone complex $M_{0, w}^\trop$, which is a moduli space for $w$-stable tropical curves of genus $0$, can only be embedded as a balanced fan in a real vector space when $w$ consists of only \textit{heavy} and \textit{light} entries, in the sense of the following definition.
	
	\begin{defn}
		Let $w = (w_i) \in \Q^n$ be a vector of rational weights satisfying $\sum w_i > 2$. We say $w_i$ is \textit{heavy} in $w$ if $w_i + w_j > 1$ for all $j \neq i$, and we say $w_j$ is \textit{light} in $w$ if $w_i + w_j > 1$ implies that $w_i$ is heavy. We say the vector $w$ is \textit{heavy/light} if it consists of only heavy and light weights.
		We say that $\overline{M}_{0, w}$ is heavy/light if the weight vector $w$ is heavy/light.
	\end{defn}
	By pursuing a construction analogous to the work of ~\cite{GMequations} for the spaces $M_{0, w}$, Cavalieri et al. prove the following theorem.
	\begin{thm}[\cite{CHMR}]\label{TropicalCompactificationM0w}
		The cone complex $M_{0, w}^\trop$ can be embedded as a balanced fan in a vector space if and only if $w$ is heavy/light. For such $w$, there is a torus embedding
		\[M_{0, w} \hookrightarrow T_w := T^{{n \choose 2} - {n - m \choose 2} - n} \]
		whose tropicalization $\Sigma_{w}$ has underlying cone complex $M_{0, w}^\trop$. The toric variety $X(\Sigma_{w})$ has torus $T_w$, and taking the closure of $M_{0, w}$ in this toric variety yields Hassett's original compactification $\overline{M}_{0, w}$.
	\end{thm}
	The fans $\Sigma_{w}$ have explicit descriptions as Bergman fans of graphic matroids and as projections of the fan $\Sigma_{n}$. In particular, these theorems exhibit the spaces $\overline{M}_{0, w}$ as \textit{wonderful compactifications} of hyperplane arrangement complements. 
	\subsection{Wonderful compactifications of hyperplane arrangement complements} Given a central arrangement of hyperplanes (codimension-one linear subspaces) $\mathcal{A}$ in $\C^n$, it is interesting to ask how the geometry and topology of the arrangement complement \[Y = \C^n \setminus \cup_{H \in \mathcal{A}} H \]
	depends on the combinatorics of $\mathcal{A}$. In ~\cite{dCP}, de Concini and Procesi study normal crossing compactifications $Y \subseteq \overline{Y}_{\mathcal{G}}$,
    termed \textit{wonderful models} or \textit{wonderful compactifications} in the literature, which depend on a choice of building set $\mathcal{G}$ for the intersection lattice $\mathcal{L}_{\mathcal{A}}$ defined by the members of $\mathcal{A}$ and all of their intersections. They also compute the cohomology ring $H^*(\overline{Y}_{\mathcal{G}}; \Z)$, showing that it only depends on $\mathcal{G}$ and $\mathcal{A}$.\\
\indent	The connection to tropical geometry is summarized as follows: a hyperplane arrangement complement $Y = \C^n \setminus \mathcal{A}$ has an embedding into an intrinsic torus $\iota: Y \hookrightarrow T = \Hom(\C^\times, Y)$, and a choice of building set for the intersection lattice of $\mathcal{A}$ is the same as a choice of building set $\mathcal{G}$ for the lattice of flats of the associated realizable matroid on the normal vectors of hyperplanes in $\mathcal{A}$. Maclagan and Sturmfels ~\cite{MStropical} describe how this choice of building set determines a fan structure $\Sigma_{\mathcal{G}}$ on the tropicalization of the torus embedding $Y \hookrightarrow T$, and hence a compactification $\overline{\iota(Y)} = \overline{Y}_{\mathcal{G}}$ by taking the closure of $\iota(Y)$ in the toric variety $X(\Sigma_{\mathcal{G}})$. Using the work of ~\cite{FYatomic} and ~\cite{Tevelev}, they conclude the following theorem.
	\begin{thm}[\cite{MStropical}]\label{ChowHA}
	The pullback $\iota^*$ induces an isomorphism of Chow rings
	\[ A^*(X(\Sigma_{\mathcal{G}})) \cong A^*(\overline{Y}_\mathcal{G}).  \]
	where $\overline{Y}_{\mathcal{G}}$ is the wonderful compactification of the hyperplane arrangement $Y$ associated to the building set $\mathcal{G}$.
	\end{thm}
	In particular, the work of ~\cite{CHMR} implies that Hassett's compactification $\overline{M}_{0, w}$ is a wonderful compactification of $M_{0, w}$, seen as a hyperplane arrangement complement when $w$ is heavy/light; specifically, the associated matroid is the graphic matroid on the \textit{reduced weight graph} $G(w)$ corresponding to $w$, and the building set giving the particular compactification $\overline{M}_{0, w}$ is the building set of \textit{1-connected flats} of the graphic matroid of $G(w)$. We will use this perspective, in combination with Theorem ~\ref{ChowHA} to prove Theorem ~\ref{Main}, thus computing $A^*(\overline{M}_{0, w})$ whenever $w$ is heavy/light. \\
	
\subsection{Acknowledgements} We would like to express thanks to Dhruv Ranganathan for suggesting this project and many detailed discussions, to Melody Chan for guidance in learning tropical geometry, and to Brendan Hassett for several helpful conversations. We also thank the referees for useful comments.  
	
	\section{$M_{0, w}^\mathrm{trop}$ as the Bergman fan of a graphic matroid.}
 Our computation of the Chow ring begins with an analysis of the graphic matroid $M(w)$ obtained from the reduced weight graph $G(w)$. We begin by recalling some definitions from matroid theory.
	\subsection{Preliminaries on matroids.}
	Matroids are combinatorial objects abstracting the notion of linearly independent subsets of a vector space. For a more comprehensive treatment of matroids in algebraic geometry, the reader is invited to consult ~\cite{Katz}.
\begin{defn}
		A \textit{finite matroid} $M$ is a pair $(E, I)$ where $E$ is a finite set, called the \textit{ground set}, and $I$ is a collection of subsets of $E$, called \textit{independent sets}, satisfying
		\begin{enumerate}[(i)]
			\item $\varnothing \in I$,
			\item if $X \in I$ and $Y \subseteq X$, then $Y \in I$, and
			\item if $X, Y \in I$ with $|X| > |Y|$, then there exists $e \in X \setminus Y$ such that $Y \cup \{e \} \in I$. 
		\end{enumerate}	
\end{defn}
The easiest examples of finite matroids are constructed from finite subsets of vector spaces. If $V$ is a vector space, and we have a finite collection of vectors $E =\{v_1, \ldots, v_n \} \subseteq V$, we can take the collection $I$ of independent sets of subsets to be linearly independent subsets of $E$. Matroids which arise in this way are said to be \textit{realizable}. This procedure shows how one associates a realizable matroid to any hyperplane arrangement $\mathcal{A}$ in $\C^n$: one takes the ground set $E$ to be the set of normal vectors of the hyperplanes in $\mathcal{A}$. Just as one can define a notion of rank or dimension for subsets of vector spaces, one has a definition of a rank function for subsets of the ground set of a matroid.
	\begin{defn}
		Let $M$ be a finite matroid with ground set $E$. The \textit{rank function} of $M$ is the function $r: \mathcal{P}(E) \to \N$ that takes a subset $S \subseteq E$ to the size of the maximal independent subset of $S$.
	\end{defn}
		The rank function allows us to define a closure operator for matroids, which leads to the notion of flats of a matroid.
	\begin{defn}
		Let $M$ be a matroid on a finite set $E$, and let $A \subseteq E$. The \textit{closure} of $A$ is defined as
		\[\mathrm{cl}(A) = \{x \in E \mid r(A \cup \{x\} ) = r(A) \}. \]
		A subset $F \subseteq E$ is called a \textit{flat} if it is closed under taking the matroid closure.
	\end{defn}
	We remark that the set of flats of a matroid forms a lattice, aptly named the \textit{lattice of flats}.
	
	We now define the notion of the graphic matroid $M(G)$ for a finite graph $G$, which is a matroid whose ground set is the edge set $E(G)$. We say a subgraph of $G$ is a \textit{forest} if it is acyclic, i.e. does not contain any cycles.
	\begin{defn}
	 Given a finite graph $G$, the \textit{graphic matroid} $M(G)$ associated to $G$ has as its ground set the edge set $E(G)$ of $G$, where a subset $S$ of $E(G)$ is independent if and only if $S$ is the set of edges of some forest of $G$.
	\end{defn}
	
	\begin{rem}
	Given a finite graph $G$ with vertex set $V(G) =\{ v_1, \ldots, v_n\}$, there is an associated hyperplane arrangement $\mathcal{A}_G$ in $\C^n$. Giving $\C^n$ coordinates $z_1, \ldots, z_n$, we let $H_{ij}$ be the hyperplane defined by $z_i - z_j = 0$. Then we may define
	\[\mathcal{A}_G := \{H_{ij} \mid i,j \text{ are connected by an edge in }G \}.  \]
	Then the realizable matroid on the normal vectors of the hyperplanes in $\mathcal{A}_G$ is the same as the graphic matroid $M(G)$: linearly independent subsets correspond exactly to forests of $G$. 
	\end{rem}

	We now introduce the Bergman fan of a matroid $M$, first in terms of the \textit{chain-of-flats} subdivision.
	\begin{defn}\label{chainofflats}
	    Given a finite matroid $M$ on the ground set $E$, the 
	    \textit{Bergman fan} $B(M)$ is a polyhedral cone complex that coincides with the order complex of the lattice of flats in $M$. 
	    More precisely, given a chain of flats 
	    \[
	    F_{\bullet} = \varnothing \subsetneq F_1 \subsetneq \cdots \subsetneq F_{r-1} \subsetneq F_r = E
	    \] where $r$ is the rank of $M$ and $F_i$ is a flat of rank $i$ in $M$ for all $i = 1, \ldots, r$, a top-dimensional cone in the Bergman fan $\Sigma(M)$ is a cone in $\R^{|E|}$ spanned by rays corresponding to ${v_{F_i} = -\sum_{j \in F_i} e_j}$ for all $i = 1, \ldots, r-1$. The polyhedral structure thus obtained is called the \textit{chain-of-flats subdivision} of the Bergman fan.
	\end{defn}
	Feichtner and Sturmfels ~\cite{FS05} demonstrate that there are multiple polyhedral structures on $|B(M)|$ (multiple fans having the same support), each corresponding to a choice of \textit{building set} for the lattice of flats. 
	\begin{defn}
	        Let $\mathcal{F}$ denote the lattice of flats of a matroid $M$, and given two flats $F, F' \in \mathcal{F}$, write $[F, F']:= \{G \in \mathcal{F} \mid F \subseteq G \subseteq F' \}$. A \textit{building set} for $\mathcal{F}$ is a subset $\mathcal{G}$ of $\mathcal{F} \setminus \{\varnothing\}$ such that the following holds:
	       For any $F \in \mathcal{F} \setminus \{\varnothing\}$, let $G_1, \ldots, G_k$ be the maximal elements of $\mathcal{G}$ contained in $F$. Then there is an isomorphism of partially ordered sets:
	        \[\varphi_F: \prod_{j = 1}^{k} [\varnothing, G_j] \to [\varnothing, F], \]
	        where the $j$th component of $\varphi_F$ is the inclusion $[\varnothing, G_j] \subseteq [\varnothing, F]$.
	       
	       A subset $\mathcal{S}$ of a building set is called \textit{nested}, if for any set of incomparable elements $F_1, \ldots, F_l$ in $\mathcal{S}$ with $l \geq 2$, the join, i.e., the least upper bound, $F_1 \vee \ldots \vee F_l$ is not an element of $\mathcal{G}$.
	\end{defn}
	
	\begin{defn}
    Given a building set $\mathcal{G}$ for the lattice of flats $\mathcal{F}$ of $M$, we may now define the \textit{nested-sets subdivision} of the Bergman fan $B(M)$ as follows: for each nested set ${\mathcal{S} = \{F_1, \ldots, F_p\} \subseteq \mathcal{G}}$, we associate the cone 
    \[\sigma_{\mathcal{S}} = \text{cone}(v_{F_1}, \ldots, v_{F_p}),\]
    where the vectors $v_{F_j}$ are as defined in Definition \ref{chainofflats}. It is shown in ~\cite{FS05} that this procedure gives a polyhedral fan with support $|B(M)|$.
    \end{defn}
    In our discussion of the Bergman fan, with either the chain-of-flats subdivision or the nested-sets subdivision, we will always consider the \textit{reduced} Bergman fan.
    \begin{defn}
        Given a Bergman fan $B(M)$ containing the lineality space $L$ spanned by the vector $(1, 1, \ldots, 1)$ in a real vector space, the \textit{reduced Bergman fan $B'(M)$} is the quotient fan $B(M)/L$.
    \end{defn}
    
    We are now ready to discuss the embedding of the tropical moduli space $M_{0, w}^\trop$ as the Bergman fan of a graphic matroid.
    \subsection{The graphic matroid $M(w)$ of the reduced weight graph $G(w)$}
	The reduced weight graph $G(w)$ for a heavy/light weight vector $w = (1^{(m)}, \epsilon^{(n - m)})$ is constructed as follows. 
	\begin{defn}
	   Suppose $w = (1^{(m)}, \epsilon^{(n - m)})$ is a heavy/light weight vector. The \textit{reduced weight graph} $G(w)$ of $w$ has vertices $\{2, \ldots, n\}$, where vertex $i$ is connected to vertex $j$ if and only if $w_i + w_j > 1$.
	\end{defn}
	Thus the reduced weight graph $G(w)$ has $m-1$ vertices corresponding to heavy weights, labelled with numbers $2$ through $m$. We have $n - m$ vertices corresponding to light weights, which we label $m + 1$ through $n$. We connect with an edge any two vertices whose corresponding weights sum to greater than $1$, so the vertices $2$ through $m$ will form a complete graph $K_{m - 1}$. Then, each heavy vertex is connected to each of the light vertices. Figure \ref{fig:reduced-weight-example} depicts $G(w)$ when $w \in (\Q \cap [0, 1])^6$ has four heavy and two light weights.
	\begin{figure}[H]
		\centering
		\includegraphics[scale=0.5]{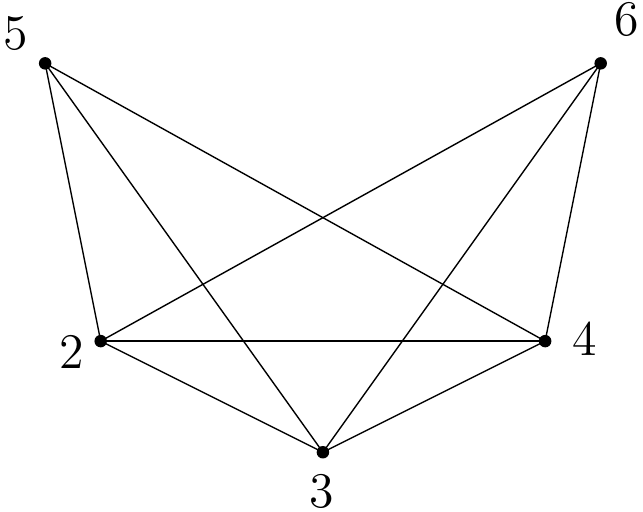}
		\caption{The reduced weight graph $G(w)$ of $w = (1^{(4)}, \epsilon^{(2)})$.}
		\label{fig:reduced-weight-example}
	\end{figure}
	For ease of notation, we put $M(w)$ for the graphic matroid of $G(w)$.
	
	\begin{defn}
	    Given a graphic matroid $M(G)$, a flat $F \subseteq E(G)$ of $M$ is said to be \textit{$1$-connected} if the subgraph of $G$ with edge set $F$ is a connected subgraph of $G$.
	\end{defn}

	\begin{thm}[\cite{CHMR}]\label{sigma0w}
	    The set $\mathcal{G}$ of $1$-connected flats of the graphic matroid $M(w)$ forms a building set for the lattice of flats of $M(w)$, and the nested-sets subdivision of $B'(M(w))$ with respect to the building set $\mathcal{G}$ embeds the cone complex $M_{0, w}^\trop$ as a balanced fan.
	\end{thm}
	
	Andy Fry ~\cite{Fry} has partially generalized the work of ~\cite{CHMR}, associating a tropical moduli space $M_{0, \Gamma}^{\trop}$ to an arbitrary graph $\Gamma$. When $\Gamma$ is the reduced weight graph of a weight vector $w$, one recovers the tropical Hassett space $M_{0, w}^\trop$.
	
	We will use the notation $\Sigma_{w}$ for the nested-sets subdivision of $B'(M(w))$ with respect to the building set of $1$-connected flats. In order to further describe the fan $\Sigma_{w}$, we now prove that the $1$-connected flats of $M(w)$ are parameterized by certain subsets of $\{2, \ldots, n \}$, which, as we will see in Section \ref{Divisors}, correspond bijectively to \textit{combinatorial types} of $w$-stable trees with one node.

	\begin{prop}\label{FlatsandSubsets}
		A $1$-connected flat of the graphic matroid $M(w)$ is uniquely determined by a set of heavy vertices $\{k_1, \ldots, k_\ell \}$ and a set of light vertices $\{s_1, \ldots, s_r\}$ in $G(w)$. There is thus a bijection between $1$-connected flats of $M(w)$ and subsets $S \subseteq \{2, \ldots, n\}$ with ${\sum_{i \in S} w_i > 1}$.
	\end{prop}
	
	\begin{proof}
		Suppose we have a connected subgraph $T$ of $G(w)$ which contains vertices $k_1, \ldots, k_\ell$ and $s_1, \ldots, s_r$. Then, for any two heavy vertices $k_i$ and $k_j$, there is a path $P$ between $k_i$ and $k_j$, so adding the edge $(k_i, k_j)$ to $T$ creates a cycle and does not increase the matroid rank (i.e. the number of edges in a spanning tree), that is, we have $r(T \cup (k_i, k_j)) = r(T)$. Similarly, there is a path from $k_i$ to $s_j$, so including the edge $(k_i, s_j)$ does not increase the matroid rank of $T$. Therefore, if $E(T)$ is to be closed under the matroid closure operator, it must contain the complete graph on the vertices $k_1, \ldots, k_\ell$ and the edges $(k_i, s_j)$ for all $i$ and $j$. The described graph is the maximal subgraph of $G(w)$ containing the relevant vertex sets, and as such must be closed: adding a new edge would add a new vertex, and thus increase the size of a spanning tree. Therefore, a subset $S \subseteq \{2, \ldots, n\}$ with $\sum_{i \in S} w_i > 1$ uniquely determines a $1$ -connected flat: simply take the maximal subgraph on the vertices indicated by $S$. Conversely, given a flat $F$, we can associate the subset $S \subseteq \{2, \ldots, n \}$ consisting of the vertices in $F$, which gives the desired bijection.
	\end{proof}	
    
\subsection{Cones of the fan $\Sigma_{w}$}\label{descriptionofSigma0w} By Proposition \ref{FlatsandSubsets}, rank-one flats of $M(w)$ correspond to subsets $\{i, j\}$ of $\{2, \ldots, n\}$ with $i < j$ and $i \leq m$. To each flat of rank one, we associate a basis vector $v_{i, j}$, and quotient by the linear relation
	\[\sum_{\substack{\{i, j\} \subseteq \{2, \ldots, n \} \\ i < j,\, i \leq m }} v_{i, j} = 0. \]
	In general, to a $1$-connected flat $F_{S}$, corresponding to a subset $S$ as in Proposition \ref{FlatsandSubsets}, we associate the vector
	\[v_{F_{S}} = \sum_{\substack{\{i, j \} \subseteq S \\ i < j,\, i \leq m}} v_{i, j}. \]
	Note that when we consider $S = \{2, \ldots, n\}$, $v_{F_{S}}$ lives in the lineality space, and because we take a quotient by the lineality space in the construction of $\Sigma_{w}$, it suffices to consider subsets $S \subsetneq \{2, \ldots, n\}$. Because $\Sigma_{w}$ is defined as the nested-sets subdivision of the Bergman fan of the graphic matroid associated to $G(w)$, it has the cone 
	\[\sigma_{F_{S_1}, \ldots, F_{S_k}} = \operatorname{cone}(v_{F_{S_1}}, \ldots, v_{F_{S_k}})  \]
	whenever $F_{S_1}, \ldots, F_{S_k}$ form a nested set within the building set of $1$-connected flats. As discussed in ~\cite{CHMR}, nested subsets in this context are either collections of flats which are pairwise vertex disjoint, or collections of flats that form chains with respect to inclusion. Now that we have described the fan $\Sigma_w$, we are ready to prove our main theorem.

\section{Computing the Chow ring from the fan $\Sigma_{w}$.}\label{Divisors}
		Let $\overline{M}_{0, w}$ be a heavy/light Hassett space with at least two heavy weights and two light weights. We continue to assume without loss of generality that $w = (1^m, \epsilon^{n - m})$ where $\epsilon < \frac{1}{n - m}$. In this section we prove Theorem ~\ref{Main}, restated here for the reader's convenience.
		\begin{customthm}{1.3}
		Let $m \geq 2$ and $n \geq 4$, and suppose that $w$ is a heavy/light weight vector, with $m$ heavy and $(n - m)$ light weights. Then the Chow ring of $\overline{M}_{0, w}$ is given as follows:
		\[A^*(\overline{M}_{0, w}) = \frac{\Z\left[D^S \mid S \subsetneq \{2, \ldots, n \},\, \sum_{i \in S} w_i > 1\right]}{\langle\text{the following relations}\rangle} \]
		\begin{enumerate}
			\item $D^S D^T = 0$ unless one of the following hold: $S \subseteq T$, $T \subseteq S$, $S \cap T = \varnothing$.
			\item For any pair of two-element subsets $\{i, j\}, \{k, \ell\} \subseteq \{2, \ldots, n \}$ with $i, k \leq m$, we have the linear relation
			\[ \sum_{\substack{S \not\supseteq \{k, \ell \}\\ S \supseteq \{i, j\}}} D^S = \sum_{\substack{S \supseteq \{k, \ell\} \\ S \not\supseteq \{ i, j \}}} D^S. \]
		\end{enumerate}
	\end{customthm}
	To prove Theorem ~\ref{Main}, we use our description of the fan $\Sigma_{w}$ together with the results of Theorem \ref{TropicalCompactificationM0w} and Theorem \ref{ChowHA} to first compute the Chow ring of $\overline{M}_{0, w}$ in terms of $1$-connected flats of the graphic matroid $M(w)$. This is the content of the following theorem.
	
	\begin{proof}
	The fan $\Sigma_{w}$ coincides with the fan $\Sigma_{\mathcal{G}}$ in the notation of Theorem \ref{ChowHA}, where $\mathcal{G}$ is the building set of $1$-connected flats for the graphic matroid $M(w)$. Therefore the Chow ring of $\overline{M}_{0, w}$ coincides with that of the toric variety $X(\Sigma_{w})$. Given our explicit description of the fan $\Sigma_{w}$ in Section \ref{descriptionofSigma0w}, the Chow ring will be a quotient of the polynomial ring
		\[\Z\left[D^S \mid S \subsetneq \{2, \ldots, n \}, \sum_{i \in S} w_i > 1 \right] \]
		by an ideal of relations; see ~\cite{Fulton} or ~\cite{cox} for a treatment of toric intersection theory. The multiplicative relations, or \textit{Stanley Reisner relations} (cf. Theorem 6.7.1, ~\cite{MStropical}) are given by setting products of the $D^S$ equal to $0$ whenever the corresponding $v_{F_{S}}$ do not span a cone. It is then immediate that the multiplicative relations are given by setting all pairs $D^S D^T = 0$, unless $S \cap T = \varnothing$, $S \subseteq T$, or $T \subseteq S$ (i.e., all the possibilities that would lead to $\{F_S, F_T \}$ forming a nested subset of $\mathcal{G}$).\\
		\\
		Computing the linear relations is a more delicate exercise. Note that because we mod out by the lineality space in the construction of $\Sigma_{w}$, the dimension of our vector space is one less than the number of two-element subsets of $\{2, \ldots, n\}$ containing at least one heavy weight; therefore, for each $\{k, \ell\} \subseteq \{2, \ldots, n\}$ such that $k \leq m$, $k< j$, the set
		\[ \mathcal{B}_{k, \ell} = \{v_{i, j} \mid \{i, j\} \subseteq \{2, \ldots, n \},\, i \leq m, i < j, \{i, j \} \neq \{k,\ell\} \} \]
		forms a basis for the ambient space. Then, for each $v \in \mathcal{B}_{k, \ell}$, we have the linear relation
		\[ \sum_{S} \langle v , v_{F_{S}}  \rangle D^S = 0, \]
		where the sum is taken over all $S \subsetneq \{2, \ldots, n \}$ with $\sum_{i \in S} w_i > 1$. For a subset $S \subsetneq \{2, \ldots, n \}$ with $\sum_{i \in S} w_i > 1$ corresponding to a flat, we can write
		\[v_{F_{S}} = \sum_{\substack{\{i, j \} \subseteq S, \\ i < j, \\i \le m}} v_{i, j} = -\sum_{\substack{\{ i,j\} \not\subseteq S,\\ i < j \\ i\le m}} v_{i, j}. \]
		The second equality is obtained because \[\sum_{\substack{\{i, j\} \subseteq \{2, \ldots, n\}\\ i < j \\ i\le m }} v_{i, j} = 0.\]
		For each basis vector $v_{i,j} \in \mathcal{B}_{k, \ell}$ with $i \leq m$, $i < j$, we have
		\begin{align*}
		0 = \sum_{S} \langle v_{i, j}, v_{F_S}\rangle D^S &= \sum_{\substack{S \not\supseteq \{k, \ell\}}}\langle v_{i, j}, v_{F_S}\rangle D^S + \sum_{\substack{S \supseteq \{k, \ell\}}}\langle v_{i, j}, v_{F_S}\rangle D^S \\&= \sum_{\substack{S \not\supseteq \{k, \ell\}, \\ S \supseteq\{i, j\}}} D^S - \sum_{\substack{S \supseteq \{k, \ell\}, \\ S \not\supseteq\{i, j\}}} D^S.
		\end{align*}
		This completes the proof of Theorem \ref{Main}.
	\end{proof}
		\subsection{Presentation of $\Sigma_{w}$ using combinatorial types of $w$-stable trees}
	\label{sec:combo-rep}
	While Theorem ~\ref{Main} computes the Chow ring for any heavy/light Hassett space, the presentation may seem unwieldy at first, in that it is in terms of subsets of $\{2, \ldots, n\}$ satisfying some restrictions. However, we may view these subsets as specifying \textit{combinatorial types} of $w$-stable trees with one node.
	\begin{defn}
	Let $w \in \left(\Q \cap(0, 1]\right)^n$, and suppose $C$ is a $w$-stable tree as in Definition \ref{Trees}. Then the \textit{combinatorial type}, or \textit{dual graph}, of $C$ is a graph that has a vertex for each irreducible component of $C$, an edge between two vertices when the corresponding components share a node, and labelled rays recording the distribution of the marked points across the components.
	\end{defn}
	Given a flat $F_S$, where $S \subseteq \{2, \ldots, n \}$ satisfies $|S|\geq 2$ and $S \cap \{2, \ldots, m \} \neq \varnothing$, we assign a $w$-stable dual graph as follows: take the complete graph $K_2$, and then to one vertex of the edge attach $|S|$ rays labelled with the elements of $S$, and then to the other edge attach $|S^c| + 1$ rays, labelled with the elements of $S^c$ and $1$. We illustrate this in Figure \ref{dualgraph}, for $w = (1, 1, \epsilon, \epsilon, \epsilon)$ and $S = \{2, 3\}$.
	\begin{figure}[h]
		\centering
		\includegraphics[scale=0.6]{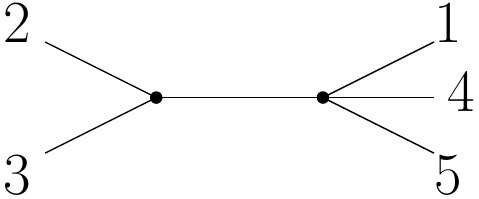}
		\caption{The combinatorial type of $(1^2, \epsilon^3)$-stable tree corresponding to the subset $S = \{2, 3\}$.}
		\label{dualgraph}
	\end{figure}
	There is thus a bijection between $1$-connected flats of $G(w)$ and combinatorial types of $w$-stable trees with precisely one node; these correspond to the codimension-$1$ boundary strata of $\overline{M}_{0, w}$. We can also parameterize $w$-stable combinatorial types by subsets $T \subseteq \{1, \ldots, n \}$ with
	$\sum_{i \in T} w_i > 1 \quad \mbox{and}\quad \sum_{i \in T^c} w_i > 1$,
	where the combinatorial type corresponding to a subset $T$ has a dual graph with rays labelled by elements of $T$ attached to one vertex, and rays corresponding to elements of $T^c$ attached to the other vertex. With this perspective, we note that the tree determined by $T$ is the same as the tree determined by $T^c$.
	
	\subsection{Example: Chow ring of Losev-Manin space $\overline{M}_{0, w}$ with $w = (1^2, \epsilon^3)$} When ${w = (1, 1, \epsilon, \ldots, \epsilon)}$, the Hassett space $\overline{M}_{0, w}$ is called the \textit{Losev-Manin} moduli space, introduced in ~\cite{LosevManin}. We will now compute the Chow ring of the Losev-Manin space $\overline{M}_{0, w}$ with weight vector $(1^2, \epsilon^3)$.
	    The reduced weight graph $G(w)$ is the three-edge star in Figure \ref{fig:losev-manin5-reduced-graph}.
	    
	    The $1$-connected flats of the graphic matroid $M(w)$ are all the copies of $1$-edge and $2$-edge subtrees of $G(w)$, as shown in Table \ref{tab:losev-manin-flats}. 
	    
	    \begin{figure}[H]
	    \begin{center}
	    \begin{tikzpicture}[scale=1.2]
		\draw[black, thin] (0, 1) -- (1, 1);
		\draw[black, thin] (0, 0) -- (0, 1);
		\draw[black, thin] (0, 1) -- (1, 0);
		\filldraw[black] (1,0) circle (1pt) node[anchor=north] {$4$};
		\filldraw[black] (0,0) circle (1pt) node[anchor=north] {$3$};
		\filldraw[black] (0,1) circle (1pt) node[anchor=south] {$2$};
		\filldraw[black] (1,1) circle (1pt) node[anchor=south] {$5$};
	    \end{tikzpicture}
	    \end{center}
	    \caption{The reduced weight graph $G(w)$ with weight vector $w = (1, 1, \epsilon, \epsilon, \epsilon)$.}
	    \label{fig:losev-manin5-reduced-graph}
        \end{figure}
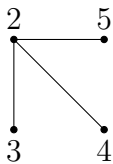
	    
	    \begin{table}[H]
		\begin{tabular}{ |c|c|c|c|} 
		\hline
 		rank 1 & $F_1$ & $F_2$ & $F_3$  \\ \hline
 		& 
		\begin{tikzpicture}
		\draw[black, thin] (0, 1) -- (0, 0);
		\filldraw[black] (1,0) circle (1pt) node[anchor=north] {$4$};
		\filldraw[black] (0,0) circle (1pt) node[anchor=north] {$3$};
		\filldraw[black] (0,1) circle (1pt) node[anchor=south] {$2$};
		\filldraw[black] (1,1) circle (1pt) node[anchor=south] {$5$};
		\end{tikzpicture}
		& 
		\begin{tikzpicture}
		\draw[black, thin] (0, 1) -- (1, 0);
		\filldraw[black] (1,0) circle (1pt) node[anchor=north] {$4$};
		\filldraw[black] (0,0) circle (1pt) node[anchor=north] {$3$};
		\filldraw[black] (0,1) circle (1pt) node[anchor=south] {$2$};
		\filldraw[black] (1,1) circle (1pt) node[anchor=south] {$5$};
		\end{tikzpicture}
		& 
		\begin{tikzpicture}
		\draw[black, thin] (0, 1) -- (1, 1);
		\filldraw[black] (1,0) circle (1pt) node[anchor=north] {$4$};
		\filldraw[black] (0,0) circle (1pt) node[anchor=north] {$3$};
		\filldraw[black] (0,1) circle (1pt) node[anchor=south] {$2$};
		\filldraw[black] (1,1) circle (1pt) node[anchor=south] {$5$};
		\end{tikzpicture}
		\\
		\hline
 		rank 2 & $F_4$ & $F_5$ & $F_6$ \\ \hline
 		& 
		\begin{tikzpicture}
		\draw[black, thin] (0, 1) -- (1, 1);
		\draw[black, thin] (0, 1) -- (0, 0);
		\filldraw[black] (1,0) circle (1pt) node[anchor=north] {$4$};
		\filldraw[black] (0,0) circle (1pt) node[anchor=north] {$3$};
		\filldraw[black] (0,1) circle (1pt) node[anchor=south] {$2$};
		\filldraw[black] (1,1) circle (1pt) node[anchor=south] {$5$};
		\end{tikzpicture}
		& 
		\begin{tikzpicture}
		\draw[black, thin] (0, 1) -- (1, 1);
		\draw[black, thin] (0, 1) -- (1, 0);
		\filldraw[black] (1,0) circle (1pt) node[anchor=north] {$4$};
		\filldraw[black] (0,0) circle (1pt) node[anchor=north] {$3$};
		\filldraw[black] (0,1) circle (1pt) node[anchor=south] {$2$};
		\filldraw[black] (1,1) circle (1pt) node[anchor=south] {$5$};
		\end{tikzpicture}
		& 
		\begin{tikzpicture}
		\draw[black, thin] (0, 1) -- (0, 0);
		\draw[black, thin] (0, 1) -- (1, 0);
		\filldraw[black] (1,0) circle (1pt) node[anchor=north] {$4$};
		\filldraw[black] (0,0) circle (1pt) node[anchor=north] {$3$};
		\filldraw[black] (0,1) circle (1pt) node[anchor=south] {$2$};
		\filldraw[black] (1,1) circle (1pt) node[anchor=south] {$5$};
		\end{tikzpicture}
		\\
		\hline
		\end{tabular}
		\caption{The $1$-connected flats of the graphic matroid $M(G(w))$ with $w = (1^2, \epsilon^3)$.}
		\label{tab:losev-manin-flats}
        \end{table}
	    
	    Now we have the lattice of flats of $M(w)$ that contains six chains of flats 
	    \[\varnothing \subsetneq F_i \subsetneq F_j \subsetneq E,\]
	    such as $\varnothing \subsetneq F_1 \subsetneq F_6 \subsetneq E$. 
	    These chains correspond to six top-dimensional cones $C_{ij}$ spanned by $v_{F_i}$ and $v_{F_j}$. 
	    Assigning the basis elements $-e_k$ to the rank-one flats $F_{k}$ for all $k = 1, \ldots, |E|$, 
	    and modding out the relation that $e_1 + e_2 + e_3 = 0$, 
	    we obtain the reduced Bergman fan $\Sigma_{w}$ in Figure \ref{fig:losev-manin-bergman} embedded in $\R^2$.
	    We have the presentation of $\Sigma_{w}$ using $w$-stable trees in Figure \ref{fig:losev-manin-combo-rep}.
	    By the Orbit-Cone Correspondence (Theorem 3.2.6, \cite{cox}) of toric divisor theory, $1$-dimensional rays $\rho^{T}$ of $\Sigma_{w}$ for $T \subseteq \{2, \ldots, n\}$ as in Theorem \ref{Main} correspond to the divisors $D^{T}$ in $X(\Sigma_{w})$.
	    Furthermore, 
	    the relations in Theorem \ref{Main} (1) are equivalent to saying that $D^S D^T = 0$ whenever $\rho^{S}$ and $\rho^{T}$ do not span a $2$-dimensional cone in $\Sigma_{w}$. Hence, reading off the nine pairs of non-adjacent one-dimensional rays, we obtain the following Stanley-Reisner relations: 
	    \begin{align*}
	        &D^{\{2, 3\}} D^{\{2, 4\}} = D^{\{2, 3\}} D^{\{2,4,5\}} = D^{\{2, 3\}} D^{\{2, 5\}} = 0, \\
	        &D^{\{2,3,5 \}} D^{\{2,4,5\}} = 
	        D^{\{2,3,5\}} D^{\{2, 4\}} = 
	        D^{\{2,3,5\}} D^{\{2,3,4\}} = 0, \\
	        &D^{\{2, 5\}} D^{\{2, 4\}} = 
	        D^{\{2, 5\}} D^{\{2,3,4\}} =
	        D^{\{2, 5\}} D^{\{2, 3\}} = 0.
	    \end{align*}
	    The linear relations in Theorem \ref{Main} (2) are
	    \begin{align*}
	        D^{\{2, 3\}} + D^{\{2, 3, 5\}} &= D^{\{2, 4\}} + D^{\{2,4,5\}}, \\
	        D^{\{2,3\}} + D^{\{2,3,4\}} &= D^{\{2, 5\}} + D^{\{2,4,5\}}, \\
	        D^{\{2,4\}} + D^{\{2,3,4\}} &= D^{\{2,5 \}} + D^{\{2,3,5\}}.
	    \end{align*}
	    Therefore, the Chow ring of $\overline{M}_{0, w}$ is 
	    \[
	    A^{\ast}(\overline{M}_{0, w}) \cong \frac{\Z[D^S \mid S \subsetneq \{2, \ldots, 5\}, \sum_{i \in S} w_i > 1 ]}{\langle \text{Stanley-Reisner relations and linear relations} \rangle}.
	    \]
	    One may verify that this coincides with the standard presentation of the Chow ring of $\P^2$ blown up at three torus-invariant points.

        \begin{figure}[H]
        \centering
        \includegraphics[scale=0.9]{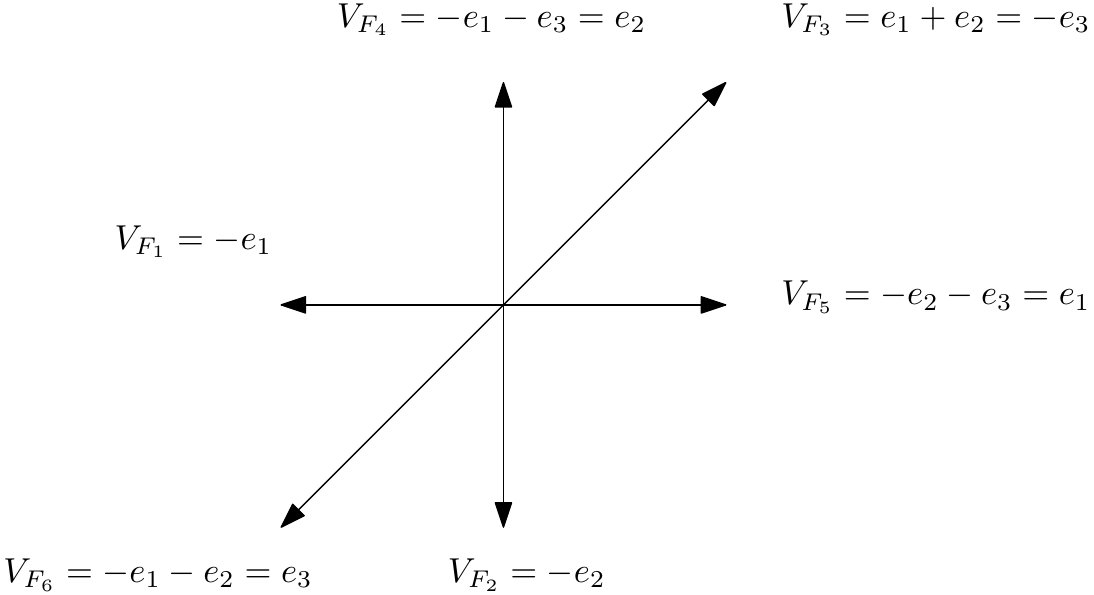}
        \caption{The reduced Bergman fan $\Sigma_{(1^2, \epsilon^{3})}$ embedded in $\R^{2}$ after taking the quotient by the lineality space.}
        \label{fig:losev-manin-bergman}
        \end{figure}
        
        \begin{center}
        \begin{figure}[H]
        \centering
        \includegraphics[scale=0.9]{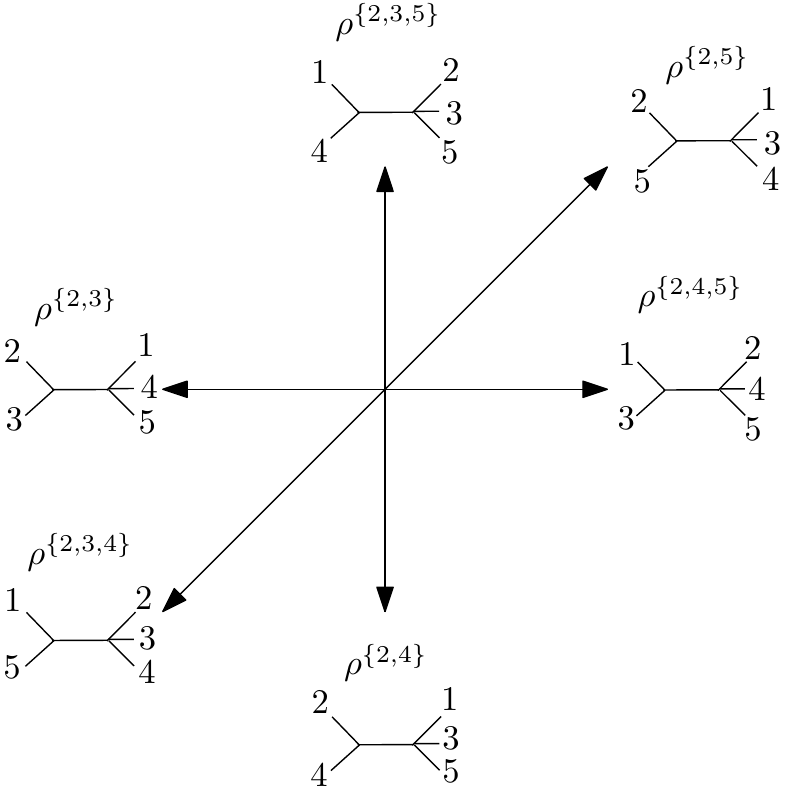}
        \caption{The presentation of $\Sigma_{(1^2, \epsilon^{3})}$ using combinatorial types of $w$-stable trees described in Section \ref{sec:combo-rep}.}
        \label{fig:losev-manin-combo-rep}
        \end{figure}
        \end{center}
	    
\section{Birational reduction morphisms}
In ~\cite{Hassett}, Hassett defines natural birational reduction morphisms $\rho_{w, w'}: \overline{M}_{0, w} \to \overline{M}_{0, w'}$ whenever $w = (w_i), w' = (w_i')$ satisfy $w_i' \leq w_i$ for all $i$; this morphism is an inclusion ${M_{0, w} \hookrightarrow M_{0, w'}}$ on the level of smooth loci, and for a pair $(C, \sum_{i} w_i P_i) \in \overline{M}_{0, w} \setminus M_{0, w}$, the morphism $\rho_{w, w'}$ collapses those components of $C$ along which $K_C + \sum_{i} w_i' P_i$ fails to be an ample divisor (more concretely, $\rho_{w, w'}$ collapses the components $T$ of $C$ for which the sum of the weights of points on $T$ together with the number of nodes on $T$ is not more than two). On the other hand, the fan $\Sigma_{n}$ is embedded in $\R^{\binom{n}{2} - n}$, with rays corresponding to corresponding to $w$-stable combinatorial types for $w = (1^{(n)})$, as discussed in the previous section. When $w = (1^{(m)}, \epsilon^{(n - m)})$ there is a natural projection \[\pr_w: \R^{\binom{n}{2} - n} \to  \R^{\binom{n}{2}-\binom{n - m}{2} - n}\]
which projects away those rays corresponding to combinatorial types becoming unstable with respect to the weight vector $w$; see ~\cite{CHMR} for an in-depth discussion. In particular, it is shown in ~\cite{CHMR} that $\pr_w(\Sigma_{n}) = \Sigma_{w}$, and that $\pr_w$ is a morphism of fans, therefore inducing a toric morphism of toric varieties $X(\pr_w): X(\Sigma_{n}) \to X(\Sigma_{w})$. In this section we prove the following theorem, by proving that $X(\pr_w)$ is an extension of Hassett's birational reduction morphism $\rho_w: \overline{M}_{0, n} \to \overline{M}_{0, w}$ to the ambient toric varieties.
\begin{thm}\label{SubringTheorem}
	Let $w$ be a heavy/light weight vector, and let $\rho_w: \overline{M}_{0, n} \to \overline{M}_{0, w}$ denote the corresponding birational reduction morphism. Then the pullback
	\[\rho_w^*: A^*(\overline{M}_{0, w}) \to A^*(\overline{M}_{0, n}) \]
	identifies the Chow ring of $\overline{M}_{0, w}$ with the subring of $A^*(\overline{M}_{0, n})$ generated by the divisor classes of $w$-stable trees.
\end{thm}

We first prove an intermediate lemma.
 
\begin{lem}
	The following diagram commutes, where $\pr_w: \R^{\binom{n}{2} - n } \to \R^{\binom{n}{2} - \binom{n - m}{2} - n}$ is the projection, and the maps $\iota, \iota_w$ are the inclusions.
	\[\begin{tikzcd}
		&\overline{M}_{0, n} \arrow[d, "\rho_w"] \arrow[r, "\iota"] & X(\Sigma_{n}) \arrow[d, "X(\pr_w)"]\\
		&\overline{M}_{0, w} \arrow[r, "\iota_w"] &X(\Sigma_{w})
	\end{tikzcd}\]
\end{lem}
\begin{proof}
	 The embedding $\overline{M}_{0, n} \hookrightarrow X(\Sigma_{n})$ is defined by taking the closure of the inclusion $M_{0, n} \hookrightarrow T^{\binom{n}{2}}/T^n$ induced by the Pl{\"u}cker embedding of the Grassmannian $\mathrm{Gr}(2, n)$. Concretely, we identify $M_{0, n}$ by the quotient of the open subset $U \subseteq \mathrm{Gr}(2, n)$ specified by the nonvanishing of Pl{\"u}cker coordinates. An $n$-tuple of distinct points $([z_{01}: z_{11}], \ldots, [z_{0n}: z_{1n}])$ in $\P^1$ can be encoded in the $2 \times n$ matrix
	 \begin{equation}\label{MatrixDescription} \begin{pmatrix}
 		z_{01} &z_{02}\cdots  &z_{0n} \\
 		z_{11} &z_{12}\cdots  &z_{1n}
	 \end{pmatrix},
	 \end{equation}
	 up to the action of $T^n$, which acts via right-multiplication by diagonal $n \times n$ matrices. The matrix has nonvanishing Pl{\"u}cker coordinates precisely because the points are distinct. Therefore, taking Pl{\"u}cker coordinates gives the embedding \[M_{0, n} \hookrightarrow T^{\binom{n}{2}}/T^n = T^{\binom{n}{2} - n} \]
	 The reader may consult ~\cite{GMequations} for a detailed discussion of this embedding.
	  The embedding $\overline{M}_{0, w} \hookrightarrow X(\Sigma_{w})$ can be concretely described: to get a matrix description as in (\ref{MatrixDescription}) for a point of $M_{0, w}$, we allow the Pl{\"u}cker coordinate corresponding to the minor indexed by columns $i$ and $j$ to vanish if and only if $w_i + w_j < 1$, tracking that these points are allowed to coincide. By forgetting the corresponding coordinates in $T^{\binom{n}{2} - n}$, we get a torus $T^{\binom{n}{2}-\binom{n-m}{2} - n}$, and we let \[\Pr_{w}: T^{\binom{n}{2} - n} \to T^{\binom{n}{2}-\binom{n-m}{2} - n}\] 
	  denote the projection map. In ~\cite{CHMR} it is proven that taking the nonvanishing Pl{\"u}cker coordinates of the matrix description of a point in $M_{0, w}$ gives an embedding \[M_{0, w} \hookrightarrow T^{\binom{n}{2}-\binom{n-m}{2} - n}.\] We claim this gives rise to a commuting square
	\[\begin{tikzcd}
	&M_{0, n} \arrow[d, "\rho_w"] \arrow[r] & T^{\binom{n}{2} - n} \arrow[d, "\Pr_w"]\\
	&M_{0, w} \arrow[r] &T^{\binom{n}{2}-\binom{n-m}{2} - n}
	\end{tikzcd}.\]
	Indeed, $\rho_w: M_{0, n} \to M_{0, w}$ is simply the inclusion of one smooth locus into the other, corresponding to viewing a $2 \times n$ matrix $A$ with nonvanishing minors as a $2 \times n$ matrix where the minor indexed by $(i, j)$ may vanish if and only if $w_i + w_j < 1$. Now to finish the proof, we need only show that the map $\Pr_w$ extends to a map $X(\Sigma_{n}) \to X(\Sigma_{w})$ of toric varieties, and that the extension equals $X(\pr_w)$. For this we need only to show that the induced map
	\[ \Hom(T, T^{\binom{n}{2} - n}) \otimes \R \to \Hom(T, T^{\binom{n}{2}-\binom{n-m}{2} - n}) \otimes \R \] 
	of real vector spaces of one-parameter subgroups, defined via composition with $\Pr_w$, induces a map of fans $\Sigma_{n} \to \Sigma_{w}$, and that the induced map agrees with the given map of fans $\pr_w$. This is the content of Lemma 3.6 in ~\cite{CHMR}, which states that the \textit{tropicalization} of the morphism $\Pr_w$ of tori agrees with $\pr_w$. The definition of tropicalization for maps of tori is simply the induced map on spaces of one-parameter subgroups. One may consult ~\cite{MStropical} for details.
\end{proof}	
\begin{proof}[Proof of Theorem \ref{SubringTheorem}.]
The maps $\iota^*$ and $\iota_w^*$ give identifications $A^*(\overline{M}_{0, n}) \cong A^*(X(\Sigma_{n}))$ and $A^*(\overline{M}_{0, w}) \cong A^*(X(\Sigma_{w}))$, respectively. In the proof of Theorem 2.17 in ~\cite{CHMR}, the authors show that when $w$ is heavy/light with at least two heavy entries, the map \[\pr_w: \Sigma_{n} \to \Sigma_{w}\] of fans contracts exactly those top-dimensional cones which parameterize combinatorial types which fail to be $w$-stable, and that $\pr_w$ is injective on all other cones. It follows that the pullback
\[X(\pr_w)^* : A^*(X(\Sigma_{w})) \to A^*(X(\Sigma_{n})) \]
is an injection (see ~\cite{cox} for a discussion of pullbacks of torus-invariant divisors under toric morphisms), identifying $A^*(X(\Sigma_{w}))$ with the subring of $A^*(X(\Sigma_{n}))$ generated by the divisors corresponding to rays not crushed by $\pr_w$. Under $\iota^*$, these are precisely the divisors indexed by $w$-stable trees on $\overline{M}_{0, n}$.
\end{proof}
	\bibliographystyle{alpha}
	\bibliography{thesis-bib}
\end{document}